\documentclass{amsart}

\usepackage[pagewise]{lineno}

\usepackage[english,spanish,portuguese]{babel}

\newenvironment{poliabstract}[1]
 {\renewcommand{\abstractname}{#1}\begin{abstract}}
 {\end{abstract}}

\usepackage{amsthm}
\usepackage{amssymb}
\usepackage{amsmath}
\usepackage{graphicx}
\usepackage{amscd}
\usepackage{amsfonts}
\usepackage{amsbsy}
\usepackage[T1]{fontenc}
\usepackage[english]{babel}

\usepackage{epsfig}
\usepackage{amssymb}

\newtheorem{thm}{Theorem}

\newtheorem{prop}{Proposition}
\newtheorem{definition}{Definition}
\newtheorem{lemma}{Lemma}

\makeatletter
\@namedef{subjclassname@2020}{%
  \textup{2020} Mathematics Subject Classification}
\makeatother

\newtheorem{ex}{Example}

\title[Sufficient conditions for rescaling expansivity]{Sufficient conditions for rescaling expansivity}

\author{A. Rojas}

\address{Instituto de Matem\'atica, Universidade Federal do Rio de Janeiro, P. O.
Box 68530, 21945-970 Rio de Janeiro, Brazil.}
\email{tchivatze@gmail.com}

\author{X. Wen}
\address{LMIB, Institute of Artificial Intelligence and School of Mathematical Science, Beihang University, Beijing, China.}
\email{wenxiao@buaa.edu.cn}

\author{Y. Yang}
\address{Department of Mathematics, Liaoming University, China.}
\email{yynmath@163.com}

\thanks{AR was partially supported by Basic Science Research Program through the NRF funded by the Ministry of Education (Grant Number: 2022R1l1A3053628)
and CNPq-Brazil grant No 307776/2019-0.
XW was partially supported by NSFC12071018 and the Fundamental Research Funds for the Central Universities.
YY was partially supported by the National Natural Science Foundation of China (Grant No. 12101281).}



\begin{document}

\selectlanguage{english}
\begin{poliabstract}{Abstract} 
We demonstrate that any k*-expansive vector field on a closed manifold exhibits rescaling expansiveness. This enhances the principal outcome outlined in \cite{a}. The verification of this assertion hinges on the introduction and exploration of the novel concept termed as a "singular-expansive flow," which will be thoroughly examined.
\end{poliabstract}


\maketitle

\section{Introduction}
\noindent
The concept of expansive flows, introduced by Bowen and Walters \cite{bw} as an extension of the notion for homeomorphisms \cite{u}, forms the basis of a vast theory. However, this theory excludes significant examples such as the geometric Lorenz attractor \cite{abs, g} and the Cherry flow \cite{pm}. Various attempts have been made by different authors to incorporate these examples into the expansive theory.
For instance, Komuro \cite{k} defined {\em k*-expansive flows} and established the k*-expansiveness of the geometric Lorenz attractor. Oka \cite{o} demonstrated the equivalence between k*-expansivity and expansivity for nonsingular flows. Araujo et al. \cite{appv} extended Komuro's work by proving the k*-expansiveness of every singular-hyperbolic attractor in three-dimensional differentiable flows. Several authors have also revealed intriguing properties of k*-expansive flows, as evident in \cite{a0, a1, brv}.
More recently, Wen and Wen \cite{ww} introduced {\em rescaling expansive flows} and demonstrated the rescaling expansiveness of multisingular hyperbolic flows \cite{bl}. Artigue \cite{a} contributed by establishing a sufficient condition, termed {\em efficiency}, for a k*-expansive flow to be rescaling expansive. Notably, this condition holds true, for instance, when the fixed points of the flow are hyperbolic. Consequently, it was concluded that $C^1$ generic k*-expansive vector fields on closed manifolds are rescaling expansive.

In this paper, we advance Artigue's work \cite{a} by proving that every k*-expansive vector field on a closed manifold exhibits rescaling expansiveness. This proof hinges on a novel notion of expansivity for flows termed {\em singular-expansivity}. Our analysis delves into the dynamics of singular-expansive flows on metric spaces. Specifically, we establish that the set of periodic orbits in a singular-expansive flow is countable. Moreover, if the singular set is dynamically isolated, the set of periodic orbits with a specified period is finite.
We also demonstrate the existence of singular-expansive flows possessing the shadowing property but lacking expansiveness. Additionally, we explore the notion of {\em singular-equicontinuous flows}, extending the classical concept of equicontinuous flows \cite{au}. Notably, we prove the existence of flows that are singular-expansive and singular-equicontinuous yet not equicontinuous. Furthermore, we establish that the Bowen entropy of nonsingular points vanishes for singular-equicontinuous flows.
Let us now precisely state our main result.

Denote by $M$ a {\em closed manifold} i.e. a compact connected boundaryless manifold
endowed with a Riemannian metric $\|\cdot\|$.
In this section we let $d$ denote the distance in $M$ induced by $\|\cdot\|$.
Denote by $B_a(x)$ the $a$-ball with center at $x$. The exponential map
of $M$ will be denoted by $exp$.

Let $V$ be a vector field of $M$ (all vector fields will be assumed to be $C^1$).
Denote by $V_t(x)$ the solution curve of the ODE $\dot{x}=V(x)$ with initial condition $x\in M$.
This produces a one-parameter family of diffeomorphisms $\{V_t:M\to M\}_{t\in\mathbb{R}}$.
We say that $\Lambda\subset M$ is {\em invariant} if $V_t(\Lambda)=\Lambda$ for every $t\in\mathbb{R}$.

\begin{definition}[\cite{ww}]
We say that $V$ is {\em rescaling expansive} on $\Lambda\subset M$ if for every $\epsilon>0$ there is $\delta>0$ such that if $x,y\in \Lambda$ satisfy
$$
d(V_t(x),V_{s(t)}(y))\leq \delta \|V(V_t(x))\|
$$
for every $t\in\mathbb{R}$ and some increasing homeomorphism $s:\mathbb{R}\to \mathbb{R}$, then $V_{s(0)}(y)\in V_{[-\epsilon,\epsilon]}(x)$.
If $V$ is rescaling expansive on $M$, we just say that $V$ is a {\em rescaling expansive vector field}.
\end{definition}

Actually this is not the original definition \cite{ww} but an equivalent one \cite{wy}.
Next we recall the notion of k*-expansive vector field
based on Komuro \cite{k}.

\begin{definition}
We say that $V$ is
{\em k*-expansive} on $\Lambda$ if
for every $\epsilon>0$ there is $\delta>0$ such that if $x,y\in \Lambda$ satisfy 
$
d(V_t(x),V_{s(t)}(y))\leq \delta
$
for every $t\in\mathbb{R}$ and some increasing homeomorphism $s: \mathbb{R}\to \mathbb{R}$ fixing $0$, then $V_{s(t_0)}(y)\in V_{[t_0-\epsilon,t_0+\epsilon]}(x)$ for some $t_0\in \mathbb{R}$.
If $V$ is k*-expansive on $X$, we just say that $V$ is k*-expansive.
\end{definition}

With these definitions we can state our main result.

\begin{thm}
\label{big}
Let $V$ be a vector field of a closed manifold $M$.
If $V$ is k*-expansive on a compact invariant set
$\Lambda\subset M$, then $V$ is rescaling expansive on $\Lambda$.
\end{thm}

The proof of this theorem is based on the following
definition. Denote by $Sing(V)=\{x\in X:V(x)=0\}$ the singular set of a vector field $V$.
The distance between $z\in M$ and $A\subset M$
is defined by $dist(z,A)=\inf_{a\in A}d(z,a)$.

\begin{definition}
\label{sing-exp}
We say that $V$ is {\em singular-expansive on $\Lambda\subset M$}
if for every $\epsilon>0$ there is $\delta>0$ such that if
$x,y\in \Lambda$ satisfy
$
d(V_t(x),V_{s(t)}(y))\leq \delta dist(V_t(x),Sing(V))
$
for every $t\in\mathbb{R}$ and some increasing homeomorphism $s: \mathbb{R}\to \mathbb{R}$, then $V_{s(t_0)}(y)\in V_{[t_0-\epsilon,t_0+\epsilon]}(x)$ for some $t_0\in\mathbb{R}$.
If $V$ is singular-expansive on $X$, we just say that $V$ is a {\em singular-expansive flow}.
\end{definition}

Theorem \ref{big} is clearly a direct consequence of the following two propositions.

\begin{prop}
\label{p1}
Let $V$ be a vector field of a closed manifold $M$.
If $V$ is k*-expansive on
$\Lambda\subset M$, then $V$ is singular-expansive on $\Lambda$.
\end{prop}

\begin{prop}
\label{thA}
Let $V$ be a vector field of a closed manifold $M$.
If $V$ is singular-expansive on a compact invariant set
$\Lambda\subset M$, then $V$ is rescaling expansive on $\Lambda$.
\end{prop}

This paper is organized as follows.
In Section \ref{sec2} we will prove Proposition \ref{thA}. In Section \ref{sec3} we will extend the notion of singular-expansivity from vector fields on closed manifolds to flows on metric spaces and
prove Proposition \ref{p1}.
In Section \ref{secc3} we state some topological properties of the singular-expansive flows.
These properties will be proved in Section \ref{sec4}.

The first author thanks Beihang University, Beijing, Peoples Republic of China, for its kindly hospitality during the preparation of this work.

\section{Proof of Proposition \ref{thA}}
\label{sec2}

\noindent
We first recall a flowbox theorem given in \cite{ww}. Let $V$ be a $C^1$ vector field on a closed manifold $M$. Let $x\in M\setminus Sing(V)$. Let
$$N_x(r)=\{v\in T_x M: v\bot X(x), \|v\|<r\},$$
$$ U_x(r)=\{v+tX(x): v\in N_x(r\|V(x)\|), t\in\mathbb{R}, |t|<r\}.$$
Denote by
$$F_x: U_x\to M, \ \ \ \ F_x(v+tX(x))=\varphi_t(\exp_x(v)).$$
The {\em conorm} (or mininorm) of a linear operator $L$ is defined by
$$
m(L)=\inf_{\|v\|=1}\|L(v)\|.
$$
\begin{lemma}[Proposition 2.2 in \cite{ww}]\label{flowbox}
Let $V$ be a $C^1$ vector field on $M$. There exists a constants $r_0>0$ such that for any $x\in M\setminus Sing(V)$, $F_x:U_x(r_0)\to M$ is an embedding and $\|D_pF_x\|<3$ and $m(D_pF_x)>1/3$ for any $p\in U_x(r_0)$.
\end{lemma}

Since $M$ is compact, there is $a>0$ such that
$$\|D_p\exp_x\|<3/2 \ \ \ \text{    and    }\ \ \ \ \|(D_p\exp_x)^{-1}\|<3/2$$
for all $x\in M$ and $p\in T_x M$ with $\|p\|<a$. Assume $V$ is a $C^1$ vector field on $M$. Also by the compactness of $M$, we can find $L>0$ such that for any $x\in M$, the vector field
$$\overline{V}=(\exp^{-1}_x)_*(V|_{B_a(x)})$$
is Lipschitz with Lipschitz constant $L$. We call $L$ the {\it local Lipschitz constant }of $V$. The following is an easy lemma.

\begin{lemma}\label{controlofnorm}
Let $V$ be a $C^1$ vector field on $M$. There is $c>0$ such that for any $x\in M\setminus Sing (V)$, if $d(y,x)<c\|V(x)\|$, then
$$\frac{1}{2}\|V(x)\|\leq \|V(y)\|\leq 2\|V(x)\|.$$
\end{lemma}
\begin{proof}
Since there is an upper bound of $\|V(x)\|$ by the compactness of $M$, we can find $c>0$ such that $c\|V(x)\|<a$ for all $x\in M$ at first. Let $L$ be a local Lipschitz constant of $V$, we choose $c>0$ such that $c\|V(x)\|<a$ and $c\leq \frac{1}{4L}$. Then for any $x\in M\setminus Sing(V)$ and any $y\in M$ with $d(x,y)<c\|V(x)\|$, denoting by $\overline{V}=(\exp^{-1}_x)_*(V|_{B_a(x)})$, we have
$$
\|\overline{V}(\exp_x^{-1}(y))-\overline{V}(\exp_x^{-1}(x))\|\leq L d(y,x)\leq Lc\|V(x)\|\leq \frac{1}{4}\|V(x)\|.
$$
Note that $\overline{V}(\exp_x^{-1}(x))=V(x)$. Thus we have
$$\frac{3}{4}\|V(x)\|\leq\|\overline{V}(\exp_x^{-1}(y))\|\leq \frac{5}{4}\|V(x)\|.$$
Hence
$$
\|V(y)\|=\|(D_{exp_x^{-1}(y)}\exp_x)^{-1}(\overline{V}(\exp_x^{-1}(y)))\|\leq \frac{3}{2}\cdot\frac{5}{4}\|V(x)\|<2\|V(x)\|,
$$
$$
\|V(y)\|=\|(D_{exp_x^{-1}(y)}\exp_x)^{-1}(\overline{V}(\exp_x^{-1}(y)))\|\geq \frac{2}{3}\cdot\frac{3}{4}\|V(x)\|\geq\frac{1}{2}\|V(x)\|.
$$
This ends proof of the lemma.
\end{proof}

\begin{lemma}[Lemma 2.3 in \cite{ww}]
\label{art2}
For every $C^1$ vector field $V$ of a closed manifold $M$ there is $r_0>0$ such that
\begin{enumerate}
\item if $x\in M\setminus Sing(V)$, $0<\delta<r_0/3$
and $t\in\mathbb{R}$ satisfy
$V_{[0,t]}(x)\subset B(x,\delta \|V(x)\|)$, then
$|t|<3\delta$.
\item if $x\in M\setminus Sing(V)$, $0<\delta<r_0/3$, $|t|< r_0$ and $V_t(x)\in B(x, \delta\|V(x)\|)$, then $|t|<3\delta$.
\end{enumerate}
\end{lemma}

Note that the constants $r_0$ in Lemma \ref{flowbox} and Lemma \ref{art2} are same.
The following Lemma follows the idea of Lemma 2.4 of \cite{ww}.

\begin{lemma}\label{ininterval}
Let $V$ be a $C^1$ vector field on $M$ and $r_0$ be given as in Lemma \ref{art2}. There is $\delta_0>0$ such that for any $x\in M\setminus Sing(V), y\in M$ and any increasing homeomorphism $h:\mathbb{R}\to\mathbb{R}$, if
$$d(V_t(x), V_{h(t)}(y))<\delta_0\|V(V_t(x))\|$$
for all $t\in\mathbb{R}$, then $|h(t)-h(0)-t|<r_0/2$ for all $t\in (-r_0,r_0)$.
\end{lemma}

\begin{proof}
Let $L$ be a local Lipschitz constant of $V$ and $c$ be the constant given in Lemma \ref{controlofnorm}. By the continuous dependence of solutions with respect to initial conditions, we know that $d(V_t(x), V_t(y))\leq e^{L|t|}d(x,y)$ for any $x, y\in M$ and $t\in\mathbb{R}$. Then we know that $$e^{-L|t|}\leq\displaystyle\frac{\|V(V_t(x))\|}{\|V(x)\|}\leq e^{L|t|}$$
for any $x\in M\setminus Sing (V)$ and $t\in\mathbb{R}$.

Now we choose $\delta_0$ such that the following properties are satisfied:
$$e^{2Lr_0}\delta_0<c,\ \  6(1+2e^{2Lr_0})\delta_0<r_0/2,\ \  3(e^{Lr_0}+e^{2Lr_0})\delta_0<r_0/2, $$
$$ 9(e^{Lr_0}+e^{2Lr_0})\delta_0<c,\ \  6[3+9(e^{Lr_0}+e^{2Lr_0})]\delta_0<r_0/2. $$
Assume that $x\in M\setminus Sing(V), y\in M$ and homeomorphism $h:\mathbb{R}\to\mathbb{R}$ satisfy
$$d(V_t(x), V_{h(t)}(y))<\delta_0\|V(V_t(x))\|$$
for all $t\in\mathbb{R}$. We show that $|h(t)-h(0)-t|<r_0/2$ for any $t\in (-r_0, r_0)$

Without loss of generality, we assume that $t\in(0,r_0)$. First we consider the case of $h(t)-h(0)\leq t$. In this case we have
$$d(V_{h(t)}(y), V_{t+h(0)}(y))\leq d(V_{h(t)}(y),V_t(x))+d(V_t(x),V_{t+h(0)}(y))$$
$$< \delta_0\|V(V_t(x))\|+e^{Lt}d(x,V_{h(0)}(y))<\delta_0\|V(V_t(x))\|+\delta_0 e^{Lt}\|V(x)\|$$
$$\leq \delta_0\|V(V_t(x))\|+\delta_0 e^{2Lt}\|V(V_t(x))\|<(1+e^{2Lr_0})\delta_0\|V(V_t(x))\|.$$
In the proof of above inequalities, we have already shown that
\begin{equation}
\label{for}
d(V_t(x),V_{t+h(0)}(y))<\delta_0 e^{2Lt}\|V(V_t(x))\|<\delta_0 e^{2Lr_0}\|V(V_t(x))\|.
\end{equation}
By the assumption $e^{2Lr_0}\delta_0<c$ we know that $d(V_t(x),V_{t+h(0)}(y))<c\|V(V_t(x))\|$, hence $\|V(V_t(x))\|\leq 2\|V(V_{t+h(0)}(y))\|$. Then we have
$$d(V_{h(t)}(y), V_{t+h(0)}(y))<(1+e^{2Lr_0})\delta_0\|V(V_t(x))\|<2(1+e^{2Lr_0})\delta_0\|V(V_{t+h(0)}(y))\|.$$
By the assumption that $6(1+2e^{2Lr_0})\delta_0<r_0/2$ we know $2(1+e^{2Lr_0})\delta_0<\frac{r_0}{3}$. Note that
$$-r_0<-t<h(t)-(t+h(0))\leq 0.$$
We can get that $|h(t)-(t+h(0))|<6(1+e^{2Lr_0})\delta_0<r_0/2$ by the second item of Lemma \ref{art2}.

Now let us consider the case of $h(t)-h(0)>t$. Since $h$ is an increasing homeomorphism, there exists $s\in(0,t)$ such that $h(s)=h(0)+t$. Then
$$d(V_s(x), V_t(x))\leq d(V_s(x), V_{h(0)+t}(y))+d(V_{h(0)+t}(y), V_t(x))$$
$$=d(V_s(x), V_{h(s)}(y))+d(V_{t+h(0)}(y), V_t(x))\overset{\eqref{for}}{<}\delta_0\|V(V_s(x))\|+\delta_0e^{2Lt}\|V(V_t(x))\|$$
$$\leq\delta_0 e^{L(t-s)}\|V(V_t(x))\|+\delta_0 e^{2Lt}\|V(V_t(x))\|<(e^{Lr_0}+e^{2Lr_0})\delta_0 \|V(V_t(x))\|.$$
By the assumption that $3(e^{Lr_0}+e^{2Lr_0})\delta_0<r_0/2$ we know that $(e^{Lr_0}+e^{2Lr_0})\delta_0<\frac{r_0}{3}$. Noting that $0<t-s<t<r_0$, we have
$$t-s<3(e^{Lr_0}+e^{2Lr_0})\delta_0<r_0/2$$
from the second item of Lemma \ref{art2}. By Lemma \ref{flowbox} we know that for any $\tau'\in[s,t]$
$$d(V_{\tau'}(x),V_t(x))=d(F_{V_t(x)}((\tau'-t)V(V_t(x))), F_{V_t(x)}(0))<3|\tau'-t|\|V(V_t(x))\|$$
$$=3(t-\tau')\|V(V_t(x))\|\leq 3(t-s)\|V(V_t(x))\|<9(e^{Lr_0}+e^{2Lr_0})\delta_0\|V(V_t(x))\|.$$
By the assumption that $9(e^{Lr_0}+e^{2Lr_0})\delta_0<c$ we know
$$\|V(V_{\tau'}(x))\|\leq 2\|V(V_t(x))\|.$$

For any $\tau\in[h(0)+t, h(t)]=[h(s), h(t)]$, we can find $\tau'\in[s,t]$ such that $h(\tau')=\tau$, then we have
$$d(V_\tau(y), V_{h(t)}(y))\leq d(V_\tau(y), V_{\tau'}(x))+d(V_{\tau'}(x), V_{t}(x))+d(V_t(x), V_{h(t)}(y))$$
$$<\delta_0\|V(V_{\tau'}(x))\|+9(e^{Lr_0}+e^{2Lr_0})\delta_0\|V(V_t(x))\|+\delta_0\|V(V_t(x))\|.$$
$$\leq 2\delta_0 \|V(V_t(x))\|+9(e^{Lr_0}+e^{2Lr_0})\delta_0\|V(V_t(x))\|+\delta_0\|V(V_t(x))\|$$
$$=[3+9(e^{Lr_0}+e^{2Lr_0})]\delta_0\|V(V_t(x))\|.$$
Noting that $d(V_{h(t)}(y), V_t(x))<\delta_0\|V(V_t(x))\|<c\|V(V_t(x))\|$, we have $\|V(V_t(x))\|\leq 2\|V(V_{h(t)}(y))\|$ from Lemma \ref{controlofnorm}. Hence we have
$$d(V_\tau(y), V_{h(t)}(y))<2[3+9(e^{Lr_0}+e^{2Lr_0})]\delta_0\|V(V_{h(t)}(y))\|$$
for all $\tau\in[h(0)+t, h(t)]$. This means that
$$V_{[0, h(0)+t-h(t)]}(V_{h(t)}(y))\subset B(V_{h(t)}(y), 2[3+9(e^{Lr_0}+e^{2Lr_0})]\delta_0\|V(V_{h(t)}(y))\|)$$
By the choice of $\delta_0$ such that $6[3+9(e^{Lr_0}+e^{2Lr_0})]\delta_0<r_0/2$, we know $2[3+9(e^{Lr_0}+e^{2Lr_0})]\delta_0<\frac{r_0}{3}$.
By the first part of Lemma \ref{art2} we know
$|h(t)-h(0)-t|<6[3+9(e^{Lr_0}+e^{2Lr_0})]\delta_0<r_0/2$. This ends the proof of the lemma.
\end{proof}

\begin{lemma}\label{lem5}
Let $V$ be a $C^1$ vector field on $M$. There exists $\delta_0<\frac{r_0}{3}$ such that for any $\delta\in(0, \delta_0)$ and any $x\in M\setminus Sing(V)$ and any $y\in M$ with a increasing homeomorphism $h:\mathbb{R}\to\mathbb{R}$, if
$$d(V_t(x), V_{h(t)}(y))<\delta \|V(V_t(x))\|$$
for all $t\in\mathbb{R}$ and $V_{h(0)}(y)\in V_{(-r_0, r_0)}(x)$, then $$V_{h(t)}(y)\in V_{(-3\delta, 3\delta)}(V_t(x))\subset V_{(-r_0, r_0)}(V_t(x))$$ for any $t\in(-r_0, r_0)$.
\end{lemma}
\begin{proof}
Let $\delta_0$ be chosen as in Lemma \ref{ininterval}. Without loss of generality, we can assume that $6\delta_0<r_0$. Assume that there are $\delta\in(0, \delta_0)$ and $x\in M\setminus Sing(V)$ and $y\in M$ with an increasing homeomorphism $h:\mathbb{R}\to\mathbb{R}$ such that
$$d(V_t(x), V_{h(t)}(y))<\delta \|V(V_t(x))\|$$
for all $t\in\mathbb{R}$ and $V_{h(0)}(y)\in V_{(-r_0, r_0)}(x)$. Denote by $V_{t_0}(x)=V_{h(0)}(y)$. Note that $d(V_{h(0)}(y), x)<\delta\|V(x)\|$, hence by Lemma \ref{flowbox} we have
$$\|t_0 V(x)\|=\|F_x^{-1}(V_{t_0}(x))\|<3d(V_{t_0}(x), x)<3\delta\|V(x)\|,$$
thus $|t_0|<3\delta$. Let any $t\in(-r_0, r_0)$ be given. It is easy to see that $$V_{h(t)}(y)=V_{h(t)-h(0)}(V_{h(0)}(y))=V_{h(t)-h(0)}(V_{t_0}(x))=V_{h(t)-h(0)-t+t_0}(V_t(x)).$$ By Lemma \ref{ininterval} we know that $|h(t)-h(0)-t|<r_0/2$, hence we have $|h(t)-h(0)-t+t_0|<\frac{r_0}{2}+|t_0|<\frac{r_0}{2}+3\delta_0<r_0$. By the fact that
$$d(V_{h(t)}(y), V_t(x))<\delta\|V(V_t(x))\|$$
we know $d(V_{h(t)-h(0)-t+t_0}(V_t(x)), V_t(x))<\delta \|V(V_t(x))\|$. From the second part of Lemma \ref{art2} we know $|h(t)-h(0)-t+t_0|<3\delta$. This says that $$V_{h(t)}(y)=V_{h(t)-h(0)-t+t_0}(V_t(x))\in V_{(-3\delta, 3\delta)}(V_t(x))$$ and ends the proof of the lemma.
\end{proof}

\begin{lemma}\label{lem6}
Let $V$ be a $C^1$ vector field on $M$. There exists $\delta_0$ such that for any $\delta\in(0, \delta_0)$ and any $x\in M\setminus Sing(V)$ and any $y\in M$ with a increasing homeomorphism $h:\mathbb{R}\to\mathbb{R}$, if
$$d(V_t(x), V_{h(t)}(y))<\delta \|V(V_t(x))\|$$
for all $t\in\mathbb{R}$ and $V_{h(0)}(y)\in V_{(-r_0, r_0)}(x)$, then $$V_{h(t)}(y)\in V_{(-3\delta, 3\delta)}(V_t(x))$$ for any $t\in\mathbb{R}$.
\end{lemma}
\begin{proof}
Let $\delta_0$ be the constant chosen in Lemma \ref{lem5} and $\delta\in(0,\delta_0]$ be given. If $x\in M\setminus Sing(V)$ and $y\in M$ with a increasing homeomorphism $h:\mathbb{R}\to\mathbb{R}$ satisfy $$d(V_t(x), V_{h(t)}(y))<\delta \|V(V_t(x))\|$$
for all $t\in\mathbb{R}$ and $V_{h(0)}(y)\in V_{(-r_0, r_0)}(x)$. Then by Lemma \ref{lem5} we know that
$$V_{h(\frac{r_0}{2})}(y)\in V_{(-r_0, r_0)}(V_{\frac{r_0}{2}}x).$$
Let $\tilde{h}(t)=h(t+\frac{r_0}{2})-h(\frac{r_0}{2}), \tilde{x}=V_{\frac{r_0}{2}}x, \tilde{y}=V_{h(\frac{r_0}{2})}(y)$, then
$$d(V_t(\tilde{x}),V_{\tilde{h}(t)}(\tilde{y}))=d(V_{t+\frac{r_0}{2}}(x), V_{h(t+\frac{r_0}{2})}(y))<\delta\|V(V_{t+\frac{r_0}{2}}(x))\|=\delta\|V(V_t(\tilde{x}))\|$$
for all $t\in\mathbb{R}$. And we also have $V_{\tilde{h}(0)}(\tilde{y})\in V_{(-r_0, r_0)}(\tilde{x})$. Hence for any $t\in(-r_0, r_0)$ we have
$$V_{\tilde{h}(t)}(\tilde{y})\in V_{(-3\delta, 3\delta)}(V_t(\tilde{x})),$$
and then
$$V_{h(t+\frac{r_0}{2})}(y)\in V_{(-3\delta, 3\delta)}(V_{t+\frac{r_0}{2}}(x)).$$
This prove that $V_{h(t)}(y)\in V_{(-3\delta, 3\delta)}(V_{t}(x))$ for any $t\in[0, \frac{3}{2}r_0]$. Similarly we prove the same thing holds as $t\in[-\frac{3}{2}r_0, 0]$. By induction we can prove that for any $n\in\mathbb{N}$ we have $V_{h(t)}(y)\in V_{(-3\delta, 3\delta)}(V_{t}(x))$ for any $t\in[-\frac{n}{2}r_0, \frac{n}{2}r_0]$. Hence $V_{h(t)}(y)\in V_{(-3\delta, 3\delta)}(V_{t}(x))$ for all $t\in \mathbb{R}$.
\end{proof}

\begin{proof}[Proof of Proposition \ref{thA}]
Let $V$ be a $C^1$ vector field and $\Lambda$ be a compact invariant set of $V$.
Assume that $V$ is singular-expansive on $\Lambda$.
Since $V$ is $C^1$ (in particular Lipschitz), there is
$B>0$ such that
$$
\|V(z)\|\leq B dist(z,Sing(V)),\quad\quad\forall z\in M.
$$
Fix $\epsilon>0$. Without loss of generality, we assume that $\epsilon<r_0$. For this $\epsilon$ we let $\delta'>0$ be given by the singular-expansivity of $V$ on $\Lambda$.
Define $\delta=\min\{\frac{\delta'}B, \delta_0, \frac{\epsilon}{3}\}$ where $\delta_0$ was given as in Lemma \ref{lem6}. Suppose that
$$
d(V_t(x),V_{s(t)}(y))\leq \delta \|V(V_t(x))\|
$$
for every $t\in\mathbb{R}$ and $x\in \Lambda\setminus Sing(V)$, $y\in \Lambda$ and an increasing homeomorphism $s:\mathbb{R}\to \mathbb{R}$.

Then, $x,y\in \Lambda$ satisfy
$$
\delta \|V(V_t(x))\|\leq \frac{\delta'}B \|V(V_t(x))\|\leq \delta'dist(V_t(x),Sing(V)),
\quad\quad\forall t\in\mathbb{R}
$$
and so
$$
d(V_t(x),V_{s(t)}(y))\leq \delta'dist(V_t(x),Sing(V)),
\quad\quad\forall t\in\mathbb{R}.
$$
It follows from singular expansivity that $V_{s(t_0)}(y)=V_{\tau}(V_{t_0}(x))$ for some $\tau\in [-\epsilon,\epsilon]$ and $t_0\in\mathbb{R}$. Now let $\tilde{x}=V_{t_0}(x), \tilde{y}=V_{s(t_0)}(y), \tilde{s}(t)=s(t+t_0)-s(t_0)$, then we have
$$d(V_t(\tilde{x}),V_{\tilde{s}(t)}(\tilde{y}))=d(V_{t+t_0}(x), V_{s(t+t_0)}(y))\leq \delta \|V(V_{t+t_0}(x))\|=\delta\|V(V_t(\tilde{x}))\|,$$
and $V_{\tilde{s}(0)}(\tilde y)=V_{\tau}(\tilde{x})\in V_{(-r_0, r_0)}(\tilde{x})$. By Lemma \ref{lem6} we have
$$V_{\tilde{s}(t)}(\tilde y)\in V_{(-3\delta, 3\delta)}(V_t(\tilde{x}))$$
for all $t\in\mathbb{R}$. Thus
$$V_{{s}(t)}(y)\in V_{(-3\delta, 3\delta)}(V_t({x}))\subset V_{[-\epsilon, \epsilon]}(V_t(x))$$
for all $t\in\mathbb{R}$. Therefore, $V$ is rescaling expansive on $\Lambda$ and the proof follows.
\end{proof}

\section{Singular-expansive flows and Proof of Proposition \ref{p1}}
\label{sec3}

\noindent
In this section we will transform the notion of singular-expansivity from vector fields on closed manifolds (Definition \ref{sing-exp}) to flows on compact metric spaces.
In particular, we will prove Proposition \ref{p1}.
Previously, we present some basic definitions and facts.

Let $X$ be a metric space.
A {\em flow} of $X$ is a continuous map $\phi:\mathbb{R}\times X\to X$
satisfying $\phi(0,x)=x$ and $\phi(t+s,x)=\phi(t,\phi(s,x))$ for every $x\in X$ and $t,s\in \mathbb{R}$.
Denote by $\phi_t(x)=\phi(t,x)$ the time $t$-map.
If $I\subset \mathbb{R}$ we denote $\phi_I(x)=\{\phi_t(x):t\in I\}$. We define $\phi_{[a,b]}(x)=\phi_{[b,a]}(x)$ for $a\geq b$.
Recall that a {\em singularity} of a flow $\phi$ is a point $\sigma\in X$ such that $\phi_t(\sigma)=\sigma$ for every $t\in\mathbb{R}$.
Denote by $Sing(\phi)$ the set of singularities of $\phi$. Notice that the singularities of a vector field match with the singularities of its corresponding flow.

For the sake of comparison we will present the current notions of expansivity for flows on metric spaces.
The first one is the classical notion of expansive flow by Bowen and Walters \cite{bw}.

\begin{definition}
We say that $\phi$ is {\em expansive}
on $\Lambda\subset X$ if for every $\epsilon>0$ there is $\delta>0$ such that if $x,y\in \Lambda$ satisfy
$
d(\phi_t(x),\phi_{s(t)}(y))\leq \delta
$
for every $t\in\mathbb{R}$ and some continuous map $s: \mathbb{R}\to \mathbb{R}$ fixing $0$, then
$y\in \phi_{[-\epsilon,\epsilon]}(x)$.
If $\phi$ is expansive on $X$, we just say that $\phi$ is expansive.
\end{definition}

The second is the notion of k*-expansive flow by Komuro \cite{k}.

\begin{definition}
We say that $\phi$ is
{\em k*-expansive} on $\Lambda$ if
for every $\epsilon>0$ there is $\delta>0$ such that if $x,y\in \Lambda$ satisfy $
d(\phi_t(x),\phi_{s(t)}(y))\leq \delta
$
for every $t\in\mathbb{R}$ and some increasing homeomorphism $s: \mathbb{R}\to \mathbb{R}$ fixing $0$, then $\phi_{s(t_0)}(y)\in\phi_{[t_0-\epsilon,t_0+\epsilon]}(x)$ for some $t_0\in \mathbb{R}$.
If $\phi$ is k*-expansive on $X$, we just say that $\phi$ is k*-expansive.
\end{definition}

Now, by replacing
the vector field $V$ by a general flow $\phi$  in Definition \ref{sing-exp} we obtain the following definition.

\begin{definition}
We say that $\phi$ is {\em singular-expansive on $\Lambda$}
if for every $\epsilon>0$ there is $\delta>0$ such that if
$x,y\in \Lambda$ satisfy
$
d(\phi_t(x),\phi_{s(t)}(y))\leq \delta dist(\phi_t(x),Sing(\phi))
$
for every $t\in\mathbb{R}$ and some increasing homeomorphism $s: \mathbb{R}\to \mathbb{R}$, then $\phi_{s(t_0)}(y)\in\phi_{[t_0-\epsilon,t_0+\epsilon]}(x)$ for some $t_0\in\mathbb{R}$.
If $\phi$ is singular-expansive on $X$, we just say that $\phi$ is a {\em singular-expansive flow}.
\end{definition}

We can therefore reformulate Proposition \ref{thA} by saying that
if the flow of a vector field $V$ of a closed manifold $M$ is singular-expansive on $\Lambda\subset M$, then such a flow is also rescaling expansive on $\Lambda$. Therefore,
the property of being singular-expansive is stronger than rescaling expansive.

On the other hand, there are singular-expansive flows which are not expansive: take for instance the trivial flow namely the one for which every point is a singularity (it is clearly singular-expansive
but not expansive unless the space is finite).
In particular, singular-expansive flows which are not k*-expansive also exist.
 
The next result asserts that every k*-expansive flow is singular-expansive.

\begin{thm}
\label{alcu}
Let $\phi$ be a flow of a compact metric space $X$.
If $\phi$ is k*-expansive on $\Lambda\subset X$, then $\phi$ is singular-expansive on $\Lambda$.
\end{thm}

\begin{proof}
Fix $\epsilon>0$ and let $\delta>0$ be given by the k*-expansivity of $\phi$ on $\Lambda$.
Take $x,y\in X$ and an increasing homeomorphism $s: \mathbb{R}\to \mathbb{R}$ such that
$d(\phi_t(x),\phi_{s(t)}(y))\leq
\frac{\delta}{diam(X)} dist(\phi_t(x),Sing(\phi))$ for every $t\in\mathbb{R}$.
Since $dist(z,Sing(\phi))\leq diam(X)$ for every $z\in X$, we get
$d(\phi_t(x),\phi_{s(t)}(y)\leq \delta$ for every $t\in\mathbb{R}$.
Then, by k*-expansivity and the choice of $\delta$ we get
$\phi_{s(t_0)}(y)\in\phi_{[t_0-\epsilon,t_0+\epsilon]}(x)$ proving the result.
\end{proof}

\begin{proof}[Proof of Proposition \ref{p1}]
Just apply Theorem \ref{alcu} to the flow generated by $V$.
\end{proof}

We finish this section by proving that both the rescaling and singular-expansivity are equivalent when the vector field is invertible at the singularities. More precisely, the following result holds.

\begin{thm}
\label{co1}
Let $V$ be a $C^1$ vector field of a closed manifold $M$
such that $DV(\sigma)$ is invertible for every $\sigma\in Sing(V)$. Then, $V$ is singular-expansive on $\Lambda\subset M$ if and only if it is rescaling expansive on $\Lambda$.
\end{thm}

\begin{proof}
Clearly the following property implies that $V$ is singular-expansive on $\Lambda$:

\medskip

\begin{enumerate}
\item[(*)]
For every $\epsilon>0$ there is $\delta>0$ such that
if $x,y\in \Lambda$ satisfy
$$
d(V_t(x),V_{s(t)}(y))\leq \delta dist(V_t(x),Sing(V))
$$
for every $t\in\mathbb{R}$ and some increasing homeomorphism $s: \mathbb{R}\to \mathbb{R}$, then $V_{s(0)}(y)\in  V_{(-\epsilon, \epsilon)}(x)$.
\end{enumerate}

\medskip

On the other hand, by Proposition \ref{thA}, the singular-expansivity of $V$ on $\Lambda$ implies the rescaling expansivity of $V$ on $\Lambda$.
It remains to prove that the rescaling expansivity of $V$ on $\Lambda$ implies (*).

First note that
by Lemma 3.9 in \cite{a} there is $C>0$ such that
$$
dist(z,Sing(V))\leq C \|V(z)\|,\quad\quad\forall z\in M.
$$
Fix $\epsilon>0$ and let $\delta'$ be given by rescaling-expansivity.
Let $\delta=\frac{\delta'}C$ and suppose that
$$
d(V_t(x),V_{s(t)}(y))\leq \delta dist(V_t(x),Sing(V))
$$
for every $t\in\mathbb{R}$ and some increasing homeomorphism $s:\mathbb{R}\to \mathbb{R}$.
Then,
$$
\delta dist(V_t(x),Sing(V))\leq \frac{\delta'}C dist(V_t(x),Sing(V))\leq \delta'\|V(V_t(x))\|,
\quad\quad\forall t\in\mathbb{R}
$$
hence
$
d(V_t(x),V_{s(t)}(y))\leq \delta' \|V(V_t(x))\|,
$ for every  $t\in\mathbb{R}$
and so $V_{s(0)}(y)=V_t(x)$ for some $t\in [-\epsilon,\epsilon]$. This completes the proof.
\end{proof}

These results motivate the study of the singular-expansive flows. In the next section we will present our results in this direction.

\section{Properties of singular-expansive flows}
\label{secc3}

\noindent
In the sequel we will study the topological properties of the singular-expansive flows.
This requires some notations.

Let $\phi$ be a flow of a compact metric space $X$.
We say that $x\in X$ is a {\em periodic point} of $\phi$ if $x\notin Sing(\phi)$ but there is $t>0$ such that $\phi_t(x)=x$. The minimal of such $t's$ is called the {\em period} of $x$.
We say that $x\in X$ is {\em nonwandering} if $U\cap (\bigcap_{t\geq T}\phi_t(U))\neq\emptyset$ for every $T>0$ and every neighborhood $U$ of $x$.
The set of nonwandering points is denoted by $\Omega(\phi)$.

We say that $K\subset X$ is {\em invariant} if $\phi_t(K)=K$ for every $t\in\mathbb{R}$. In such a case we denote by $\phi|_K$ the flow restricted to $K$.
We say that $K$ is {\em dynamically isolated} if
there is a neighborhood $U$ of $K$ (called {\em isolated block}) such that
$$
K=\bigcap_{t\in \mathbb{R}}\phi_t(U).
$$

\begin{thm}
\label{thAA}
The following properties hold for every singular-expansive flow $\phi$ of a compact metric space $X$.

\begin{enumerate}

\item
The set of periodic orbits of $\phi$ is countable.

\medskip

\item
If $Sing(\phi)=\emptyset$ or consists of finitely many isolated points of $X$, then $\phi$ is expansive.

\medskip

\item
If $Sing(\phi)$ is dynamically isolated, the set of periodic orbits with period $\tau\in [0,t]$ is finite ($\forall t>0$).
\end{enumerate}
\end{thm}

Item (3) above is false without the hypothesis that $Sing(\phi)$ is dynamically isolated. Based on Subsection 3.5 in \cite{a} we obtain the following counterexample.

\begin{ex}
\label{insure}
There is a compact metric space exhibiting a singular-expansive flow with infinitely many periodic orbits of period $2\pi$.
\end{ex}

\begin{proof}
Define
$$
X=\{(0,0)\}\cup \bigcup_{n\in\mathbb{N}} \{z\in \mathbb{R}^2:\|z\|=e^{-n}\}.
$$
We have that $X$ is a compact metric space if equipped with the Euclidean metric.
Consider the flow
$\phi$ on $X$ obtained by restricting (to $X$) the flow of the vector field in $\mathbb {R}^2$ defined by $V(x,y)=(-y,x)$.
Proposition 3.21 in \cite{a} implies that $V$ is rescaling expansive on $X$. Since $DV(0,0)$ is invertible, we conclude from Theorem \ref{co1} that $\phi$ singular-expansive.
On the other hand, by direct integration we obtain
$$
\phi_t(z)=(-y\sin t+x\cos t, x\sin t+y\cos y),\quad\quad \forall z=(x,y)\in X, \quad t\in \mathbb{R}.
$$
Then, $X\setminus Sing(\phi)=X\setminus \{(0,0)\}$ consists of infinitely many periodic orbits of period $2\pi$.
\end{proof}

On the other hand, we say that $\phi$ is {\em equicontinuous} when for every $\epsilon>0$ there is $\delta>0$ such that if $x,y\in X$ and $d(x,y)\leq\delta$, then $d(\phi_t(x),\phi_t(y))\leq\epsilon$ for every $t\in\mathbb{R}$.
The equicontinuous flows have been widely studied in the literature \cite{au}. Their singular version are given below.

\begin{definition}
A flow $\phi$ of a metric space $X$ is {\em singular-equicontinuous} if for every $\epsilon>0$ there is
$\delta>0$ such that if
$x,y\in X$ and $d(x,y)\leq \delta dist(x,Sing(\phi))$, then
$d(\phi_t(x),\phi_t(y))\leq \epsilon$ for every
$t\in\mathbb{R}$.
\end{definition}

Every equicontinuous flow $\phi$ on a compact metric space is singular-equicontinuous.
Indeed,
take $\epsilon>0$ and let $\delta>0$ be given by the equicontinuity of $\phi$.
If $\delta'=\frac{\delta}{diam(X)}$ and
$d(x,y)\leq \delta'dist(x,Sing(\phi))$, then $d(x,y)\leq \delta$ and so $d(\phi_t(x),\phi_t(y))\leq \epsilon$ for every $t\in \mathbb{R}$. Therefore, $\phi$ is singular-equicontinuous.

The converse is true if there are no singularities. More precisely, every singular-equicontinuous flow without singularities on a compact metric space is equicontinuous.
It follows that
flows which are both singular-expansive and singular-equicontinuous do exist (e.g. the trivial flow).
A different example will be reported below.

First recall that if $\delta,T>0$ a $(\delta,T)$-pseudo orbit of a flow $\phi$
is a sequence $(x_i,t_i)_{i\in \mathbb{Z}}$ formed by points $x_i\in X$ and times $t_i\in\mathbb{R}$ such that
$t_i\geq T$ and $d(\phi_{t_i}(x_i),x_{i+1})\leq \delta$
for every $i\in\mathbb{Z}$.
Given $\epsilon>0$ we define
$Rep(\epsilon)$ as the set of
increasing homeomorphism $s:\mathbb{R}\to \mathbb{R}$ such that
$$
\left|\frac{s(t)-s(r)}{t-r}-1\right|\leq \epsilon\quad (\forall t\neq r).
$$
We say that $(x_i,t_i)_{i\in\mathbb{Z}}$
can be {\em $\epsilon$-shadowed} if there are $x\in X$ and
$s\in Rep(\epsilon)$ such that
$$
d(\phi_{s(t)}(x),\phi_{t-S(i)}(x_i))\leq \epsilon,
\quad\quad\forall i\in\mathbb{Z},\quad \forall S(i)\leq t<S(i+1)
$$
where
$$
S(i) = \left\{\begin{matrix}
t_0+t_1+\cdots t_{i-1} &\hbox{ if } i>0,\\
0 &\hbox{ if } i=0,\\
-t_1-t_2-\cdots-t_i &\hbox{ if } i< 0.
\end{matrix}\right.
$$
Given $\Lambda\subset X$ we say that $\phi$ {\em has the shadowing property on $\Lambda$} if for every $\epsilon>0$ there is $\delta>0$ such that
every $(\delta,1)$-pseudo orbit $(x_i,t_i)_{i\in\mathbb{Z}}$ {\em on $\Lambda$}
(i.e. $x_i\in\Lambda$ for $i\in\mathbb{Z}$)
can be $\epsilon$-shadowed.
If $\phi$ has the shadowing property on $X$ we just say that $\phi$ has the shadowing property.
This definition was abbreviated to SPOTP
(strong pseudo-orbit tracing property) in Komuro \cite{k2}.
Examples of flows with this property are the Anosov ones \cite{pm}.

\begin{thm}
\label{thank}
There is a compact metric space exhibiting a flow
with the shadowing property which is singular-expansive, singular-equicontinuous but not equicontinuous.
\end{thm}

An interesting property of the equicontinuous flows is that they have zero topological entropy.
Recall the topological entropy of a flow $\phi$ of a compact metric space $X$ defined by $h(\phi)=h(\phi,X)$ where
$$
h(\phi,K)=\lim_{\epsilon\to0}\limsup_{t\to\infty}\frac{1}t\ln r(t,\epsilon)
$$ for $K\subset X$ compact,
and $r(t,\epsilon)$ is the minimal cardinality of $F\subset K$
satisfying
$$
K\subset \bigcup_{x\in F}\{y\in X:d(\phi_s(x),\phi_s(y))\leq \epsilon, \forall 0\leq s\leq t\}
$$
($F$ is then called $(t,\epsilon)$-spanning set).
We would like to prove the same property for the singular-equicontinuous flows. Instead, we will consider
the {\em Bowen entropy} \cite{b1} of the nonsingular points namely
$$
h^*(\phi)=\sup\{h(\phi,K):K\subset X\setminus Sing(\phi)\mbox{ is compact}\}
$$
In general $h^*(\phi)\leq h(\phi)$ hence $h^*(\phi)=0$ if $\phi$ is equicontinuous.
Since $\phi$ restricted to $Sing(\phi)$ is equicontinuous, $h(\phi,Sing(\phi))=0$
we expect $h^*(\phi)=h(\phi)$.
These facts motivate the result below.

\begin{thm}
\label{gasolina}
Let $\phi$ be a flow of compact metric spaces.

\begin{enumerate}
\item
If $\Omega(\phi)\setminus Sing(\phi)$ is closed, then $h^*(\phi)=h(\phi)$.

\medskip

\item
If $\phi$ is singular-equicontinuous, then $h^*(\phi)=0$.
\end{enumerate}
\end{thm}

\section{Proof of theorems \ref{thAA} to \ref{gasolina}}
\label{sec4}

\noindent
The next lemma is closely related to Lemma 2.2 in \cite{a}.

\begin{lemma}
\label{cons}
Let $\phi$ be a flow of a compact metric space $X$. If $Sing(\phi)$ is dynamically isolated, then there is $\beta_0>0$ such that $diam(\phi_{\mathbb{R}}(x))\geq\beta_0$ for every $x\in X\setminus Sing(\phi)$.
\end{lemma}

\begin{proof}
Otherwise, there is a sequence $x_n\in X\setminus Sing(\phi)$
such that $diam(\phi_{\mathbb{R}}(x_n))\to0$ as $n\to\infty$.
Let $U$ be an isolating block of $Sing(\phi)$.
Since $diam(\phi_{\mathbb{R}}(x_n))\to0$ and $X$ is compact, we can assume that $x_n\to \sigma$ for some $\sigma\in Sing(\phi)$.
Then, $\phi_{\mathbb{R}}(x_n)\subset U$ for $n$ large so
$x_n\in \bigcap_{t\in\mathbb{R}}\phi_t(U)=Sing(\phi)$ for $n$ large which is absurd. This completes the proof.
\end{proof}

The next lemma proves that the set of singularities of a singular-expansive flow is dynamically isolated.

\begin{lemma}
\label{mcm}
Let $\phi$ be a flow of a compact metric space $X$.
If $\phi$ is expansive on $X\setminus Sing(\phi)$, then $Sing(\phi)$ is dynamically isolated.	
\end{lemma}

\begin{proof}
Suppose that $Sing(\phi)$ is not dynamically isolated.
We have three cases to be considered:

\begin{enumerate}
\item There is a sequence of periodic orbits $\{\gamma_1,\gamma_2,\cdots,\gamma_n,\cdots\}$ accumulating on $Sing(\phi)$ whose periods are bounded away from $0$.
\item There is a sequence of periodic orbits $\{\gamma_1,\gamma_2,\cdots,\gamma_n,\cdots\}$ accumulating on $Sing(\phi)$ whose periods tent to $0$ as $n\to\infty$.
\item There is a neighborhood $U_0$ of $Sing(\phi)$ such that no periodic orbit of $\phi$ is entirely contained in $U_0$.	
\end{enumerate}

In Case (1) we take $\epsilon_0=\min\{\frac{\tau}4,1\}$ where $\tau$ is a positive lower bound of the periods of the periodic orbits $\gamma_n$ ($n=1,2,\cdots$).
Fix $\delta>0$. For this $\delta$ there exists a neighborhood $U_\delta$ of $Sing(\phi)$
such that $d(\phi_t(z),z)<\delta$ for every $z\in U_\delta$ and every $t\in [0,2]$.
Since $\gamma_n$ accumulates on $Sing(\phi)$, one has $\gamma_n\subset U_\delta$ for some $n\in\mathbb{N}$. 
Take $z\in \gamma_n$ and $y=\phi_{2\epsilon_0}(z)$.
It follows from the choice of $\epsilon_0$ that $2\epsilon_0\in [0,2]$. Then, $d(\phi_t(z),\phi_t(y))=d(\phi_t(z),\phi_{2\epsilon_0}(\phi_t(z)))<\delta$ for every $t\in\mathbb{R}$. Again the choice of $\epsilon_0$ implies $2\epsilon_0<\tau$ hence $y\notin \phi_{[-\epsilon_0,\epsilon_0]}(z)$.

In Case (2) we take $\epsilon_0=1$.
In this case have that $\gamma_n\to\sigma$ for some $\sigma\in Sing(\phi)$.
Given $\delta>0$ we can choose two different periodic orbits
$\gamma_n$ and $\gamma_m$ with $d_H(\gamma_n,\gamma_m)<\delta$ ($d_H$ here is the Hausdorff distance).
Taking $z\in \gamma_n$ and $y\in\gamma_m$ we have $d(\phi_t(z),\phi_t(y))<\delta$ for every $t\in\mathbb{R}$. Since $\gamma_n$ and $\gamma_m$ are different orbits, we also have $y\notin \phi_{[-\epsilon_0,\epsilon_0]}(x)$.

In Case (3) we take $\epsilon_0=1$ once more.
Take $\delta>0$. As before there is a neighborhood $U_\delta$ of $Sing(\phi)$
such that $d(\phi_t(z),z)<\delta$ for every $z\in U_\delta$ and every $t\in [0,2]$.
Since $Sing(\phi)$ is not dynamically isolated,
there is $x\in U\setminus Sing(\phi)$ such that the full orbit of $x$ contained in $U_\delta$. The assumption in Case (3) implies that the orbit of $x$ is not periodic.
Take $y=\phi_2(x)$. Then, $d(\phi_t(x),\phi_t(y))<\delta$ for every $t\in \mathbb{R}$. Since $x$ is not periodic, we have $y\notin \phi_{[-1,1]}(x)$.

Taking $h$ as the identity of $\mathbb{R}$ in all these cases we obtain that $\phi$ is not expansive on $X\setminus Sing(\phi)$. This completes the proof.
\end{proof}

We also need the following lemma.

\begin{lemma}
\label{thB}
The properties below hold for every flow
$\phi$ of a compact metric space $X$:
\begin{enumerate}
\item
If $\phi$ is expansive on $X\setminus Sing(\phi)$, then $\phi$ is singular-expansive.

\medskip

\item
If $\phi$ is singular-expansive and $Sing(\phi)$ is open, then $\phi$ is expansive on $X\setminus Sing(\phi)$.
\end{enumerate}
\end{lemma}

\begin{proof}
First we prove Item (1).
By Lemma \ref{mcm} we have that $Sing(\phi)$ is dynamically isolated.
Take an isolating block $U$ of $Sing(\phi)$.
Fix $\Delta>0$ such that
\begin{equation}
\label{lavey}
\{z\in X: dist(z,Sing(\phi))\leq\Delta\}\subset U.
\end{equation}
Let $\epsilon>0$ and set $\epsilon'=\min\{\epsilon,\Delta\}$.
For this $\epsilon'$ we let $\delta'>0$ be given by the expansivity of $\phi$ on $X\setminus Sing(\phi)$.
Define
$\delta=\frac{\min\{\delta',\Delta\}}{diam(X)}$
and let $x,y\in X$ such that
$$
d(\phi_t(x),\phi_{s(t)}(y))\leq \delta dist(\phi_t(x),Sing(\phi))
$$
for every $t\in\mathbb{R}$ and some increasing homeomorphism $s:\mathbb{R}\to \mathbb{R}$.
Define $\hat{s}(t)=s(t)-s(0)$ for $t\in\mathbb{R}$  and $\hat{y}=\phi_{s(0)}(y)$.
Then, $\hat{s}:\mathbb{R}\to \mathbb{R}$ is a continuous map fixing $0$ such that
$$
d(\phi_t(x),\phi_{\hat{s}(t)}(\hat{y}))\leq \delta dist(\phi_t(x),Sing(\phi)),\quad\quad\forall t\in\mathbb{R}.
$$
If $x\in Sing(\phi)$, $\phi_t(x)=x\in Sing(\phi)$ for $t\in\mathbb{R}$ hence $\hat{y}=x$ and so $\phi_{s(0)}(y)=\phi_t(x)$ with $t=0\in [-\epsilon,\epsilon]$.

On the other hand,
$$
\delta dist(z,Sing(\phi))=\min\{\delta',\Delta\}\frac{dist(z,Sing(\phi))}{diam(X)}\leq \min\{\delta',\Delta\},
\quad\quad\forall z\in X.
$$
Therefore,
$$
d(\phi_t(x),\phi_{\hat{s}(t)}(\hat{y}))\leq\min\{\delta',\Delta\},\quad\quad\forall t\in\mathbb{R}.
$$
If $y\in Sing(\phi)$, then $d(\phi_t(x),y)\leq \Delta$ for every $t\in\mathbb{R}$.
This and the inclusion (\ref{lavey}) imply $x\in \bigcap_{t\in\mathbb{R}}\phi_t(U)=Sing(\phi)$ and then $\phi_{s(0)}(y)=\phi_t(x)$ with $t\in[-\epsilon,\epsilon]$ as before.
Therefore, we can assume $x,y\in X\setminus Sing(\phi)$.
Since
$$
d(\phi_t(x),\phi_{\hat{s}(t)}(\hat{y}))\leq \min\{\delta',\Delta\}\leq \delta',\quad\quad\forall t\in \mathbb{R},
$$
we conclude from the expansivity on $X\setminus Sing(\phi)$ that
$\phi_{s(0)}(y)=\hat{y}=\phi_t(x)$ for some $t\in[-\epsilon,\epsilon]$ proving Item (1).

To prove Item (2) we assume that $\phi$ is singular-expansive and that $Sing(\phi)$ is open.
Then, $X\setminus Sing(\phi)$ is closed and $\phi$ has no singularities there so by Item (ii) of Theorem 1 in \cite{bw} we just need to consider increasing homeomorphisms fixing $0$ to prove the expansivity of $\phi$ on $X\setminus Sing(\phi)$.
Since $Sing(\phi)$ and $X\setminus Sing(\phi)$ are closed disjoint, there is $c>0$ such that
$
\inf\{dist(z,Sing(\phi)):z\in X\setminus Sing(\phi)\}\geq c.
$
Now, let $\epsilon>0$ and consider $\delta'$ from the singular-expansivity of $\phi$ for this $\epsilon$.
Define $\delta=\delta'c$ and
take $x,y\in X\setminus Sing(\phi)$ such that
$
d(\phi_t(x),\phi_{s(t)}(y))\leq \delta,
$
for every $t\in \mathbb{R}$ and some increasing homeomorphism
$s: \mathbb{R}\to \mathbb{R}$ fixing $0$.
Since $x\in X\setminus Sing(\phi)$ which is invariant, $\phi_t(x)\in X\setminus Sing(\phi)$ for every $t\in \mathbb{R}$.
Then, $c\leq dist(\phi_t(x),Sing(\phi))$ for every $t\in\mathbb{R}$ and then
$$
d(\phi_t(x),\phi_{s(t)}(y))\leq \delta=\delta' c\leq \delta'dist(\phi_t(x),Sing(\phi)),
\quad\quad\forall t\in\mathbb{R}.
$$
Therefore, $y=\phi_{s(0)}(y)=\phi_t(x)$ for some $t\in [-\epsilon,\epsilon]$ proving that $\phi$ is expansive on $X\setminus Sing(\phi)$.
This completes the proof.
\end{proof}

\begin{ex}
It is natural to ask if we can remove the hypothesis that $Sing(\phi)$ is open in Item (2) of Lemma \ref{thB}.
However, this is false and a counterexample is given by the geometric Lorenz attractor.
\end{ex}

We also need the result below below.

\begin{lemma}
\label{puestico}
Let $\phi$ be a singular-expansive flow of a compact metric space $X$. Then, $\phi$
is expansive on every nonsingular compact invariant set of $\phi$.
\end{lemma}

\begin{proof}
Let $\Lambda$ be a nonsingular compact invariant set of $\phi$.
We assert that $\phi|_\Lambda$ is singular-expansive.

Fix $\epsilon>0$ and let $\delta'$ be given by the singular-expansivity of $\phi$ for this $\epsilon$.
Since $\Lambda$ is compact and nonsingular, there exists $\delta'>0$ such that if
$a,b\in \Lambda$ and $d(a,b)\leq\delta'diam(X)$, then
$d(a,b)\leq\delta dist(a,Sing(\phi))$.

Now suppose that
$x,y\in \Lambda$ and $d(\phi_t(x),\phi_{s(t)}(y))\leq\delta'dist(\phi_t(x),Sing(\phi|_\Lambda))$ for every $t\in\mathbb{R}$ and
some increasing homeomorphism $s:\mathbb{R}\to \mathbb{R}$.
Since $\phi$ is nonsingular, $Sing(\phi|_\Lambda)=\emptyset$ hence
$dist(\phi_t(x),Sing(\phi|_\Lambda))=dist(\phi_t(x),\emptyset)=diam(X)$ for every $t\in\mathbb{R}$.
Then, $d(\phi_t(x),\phi_{s(t)}(y))\leq\delta'diam(X)$ for every $t\in\mathbb{R}$.
Taking $a=\phi_t(x)$ and $b=\phi_{s(t)}(y)$ we get
$d(\phi_t(x),\phi_{s(t)}(y))\leq \delta dist(\phi_t(x),Sing(\phi))$ for all $t\in\mathbb{R}$.
and so
$\phi_{s(t_0)}(y)\in \phi_{[t_0-\epsilon,t_0+\epsilon]}(x)$ for some $t_0\in\mathbb{R}$ proving the assertion.

On the other hand, since $\Lambda$ is nonsingular, $Sing(\phi|_{\Lambda})=\emptyset$ and so $Sing(\phi|_{\Lambda})$ is open.
Then,  $\phi|_{\Lambda}$ is expansive by Lemma \ref{thB}
proving the result.
\end{proof}

This corollary motivates the question if conversely
every flow
which is expansive on every nonsingular compact invariant set
is singular-expansive. But the answer is negative by the following example.

\begin{ex}
\label{gabi}
There is a compact metric space exhibiting a flow which not singular-expansive but expansive on every nonsingular compact invariant set.
\end{ex}

\begin{proof}
Following the ideas of Example \ref{insure} we
define
$$
X=\{(0,0)\}\cup \bigcup_{n\in\mathbb{N}} \{z\in \mathbb{R}^2:\|z\|=n^{-1}\}.
$$
Again $X$ is a compact metric space if equipped with the Euclidean metric. Once more we consider the flow
$\phi$ on $X$ obtained by restricting that of the vector field in $\mathbb {R}^2$ defined by $V(x,y)=(-y,x)$ on $X$.
As in Example \ref{insure} we have that
$$
\phi_t(z)=(-y\sin t+x\cos t, x\sin t+y\cos y),\quad\quad \forall z=(x,y)\in X, \quad t\in \mathbb{R}.
$$
Then, $Sing(\phi)=\{(0,0)\}$ and $\phi_t$ is a linear isometry for every $t\in\mathbb{R}$.
Now take $\epsilon=1$. Define the sequences $z_n,z_n'\in X$ by
$z_n=(\frac{1}n,0)$ and $z_n'=(\frac{1}{n+1},0)$ for $n\in\mathbb{N}$.
Notice that $z_n$ and $z_n'$ belong to different circles of $X$ and so their orbits are different.
From this we obtain $z_n'\notin \phi_{[-1,1]}(z_n)$.
On the other hand,
$$
d(\phi_t(z_n),\phi_t(z_n'))=\|z_n-z_n'\|=\frac{1}n-\frac{1}{n+1}=\frac{1}{n(n+1)}<\frac{1}{n^2},
$$
and $dist(\phi_t(z_n),Sing(\phi))=\|z_n-(0,0)\|=\frac{1}n$
so
$$
d(\phi_t(z_n),\phi_t(z_n'))<\frac{1}n dist(\phi_t(z_n),Sing(\phi)), \quad\quad\forall t\in \mathbb{R}.
$$
Since $z_n$ and $z_n'$ belong to different circles,
$\phi$ is not singular-expansive.
Finally, since every nonsingular compact invariant set consists of finitely many periodic orbits, one has that $\phi$ is expansive on all such sets. This completes the proof.
\end{proof}

The lemma below will be used to prove Theorem \ref{thank}.

\begin{lemma}
\label{bariloche}
Let $\phi$ be a flow of a compact metric space $X$.
If for every $\epsilon>0$ there is $\delta>0$ such that
$B[x,\delta dist(x,Sing(\phi))]\subset \phi_{[-\epsilon,\epsilon]}(x)$ for every $x\in X$,
then $\phi$ is both singular-expansive and singular-equicontinuous.
\end{lemma}

\begin{proof}
Let $\epsilon>0$ and $\delta>0$ be given by the condition.
If $s: \mathbb{R}\to \mathbb{R}$ is an increasing homeomorphism
satisfying
$d(\phi_t(x),\phi_{s(t)}(y))\leq \delta dist(\phi_t(x),Sing(\phi))$ for all $t\in\mathbb{R}$, then $\phi_{s(0)}(y)\in B(x,\delta dist(x,Sing(\phi)))$ and so $\phi_{s(0)}(y)\in \phi_{[-\epsilon,\epsilon]}(x)$ by the condition.
Therefore, $\phi$ is singular-expansive.

To prove that $\phi$ is singular-equicontinuous, take $\epsilon>0$ and $\eta>0$ such that
if $y=\phi_s(x)$ with $|s|\leq\eta$, then $d(\phi_t(x),\phi_t(y))\leq \epsilon$ for every $t\in\mathbb{R}$
(c.f. p. 181 in \cite{bw}).
For this $\eta$ we take $\delta>0$ given by the condition.
Therefore, if $d(x,y)\leq \delta dist(x,Sing(\phi))$, that is
$y\in B[x,\delta dist(x,Sing(\phi))]$, then
$y=\phi_s(x)$ for some $|s|\leq \eta$ thus
$d(\phi_t(x),\phi_t(y))\leq \epsilon$ for every $t\in\mathbb{R}$ proving the result.
\end{proof}

\begin{proof}[Proof of Theorem \ref{thAA}]
Let $\phi$ be a singular-expansive flow of a compact metric space $X$.
Given $\delta>0$ we denote by $U_\delta(Sing(\phi))$ the open $\delta$-ball around $Sing(\phi)$. Define

\begin{equation}
\label{bayly}
X_\delta=\bigcap_{t\in\mathbb{R}}\phi_t(X\setminus U_\delta(Sing(\phi))).
\end{equation}

It follows that $X_\delta$ is a compact invariant set without singularities of $\phi$.
Since $\phi$ is singular-expansive, $\phi$ is expansive on $X_\delta$ by Lemma \ref{puestico}. On the other hand, as is well known \cite{bw}, the set of periodic orbits of an expansive flow is countable.
Since the periodic orbits of $\phi$ are contained in $\bigcup_{n\in\mathbb{N}}X_{\frac{1}n}$, we conclude that the set of periodic orbits of $X$ is countable. This proves Item (1).

To prove Item (2) we see that if $Sing(\phi)=\emptyset$ or consists of finitely many isolated points of $X$, then $Sing(\phi)$ is open. Therefore, $\phi$ is expansive on $X\setminus Sing(\phi)$ by Item (2) of Lemma \ref{thB}.
Since $X\setminus Sing(\phi)$ and $Sing(\phi)$ are closed disjoint, we conclude that $\phi$ is expansive. This proves Item (2).

To prove Item (3) we further assume that $Sing(\phi)$ is isolated. Fix $t>0$ and suppose by contradiction that
$\phi$ has infinitely many distinct periodic orbits $O_n$
with period $t_n\leq t$.
If $inf_{n\in\mathbb{N}}dist(O_n,Sing(\phi))>0$, then
$\bigcup_{n\in\mathbb{N}}O_n\subset X_\delta$ for some $\delta>0$ which is a contradiction since $\phi$ is expansive on $X_\delta$ (see \cite{bw}).
Then, we can assume that there is a sequence $x_n\in O_n$ and $\sigma\in Sing(\phi)$ such that $x_n\to \sigma$.
Since $\sigma\in Sing(\phi)$ and the period of $O_n$ is bounded by $t$, we have that the whole $O_n\to \sigma$ with respect to the Hausdorff metric of compact subsets of $X$.
In particular, $diam(O_n)\to 0$ contradicting Lemma \ref{cons}.
This completes the proof.
\end{proof}

\begin{proof}[Proof of Theorem \ref{thank}]
Let $X=[0,1]$ be the unit interval endowed with the Euclidean metric.
For every $\lambda\in\mathbb{R}$ we define
$\phi: \mathbb{R}\times X\to X$ by
$$
\phi_t(x)=\frac{x e^{\lambda t}}{1+x(e^{\lambda t}-1)},
\quad\quad\forall 0\leq x\leq 1, t\in\mathbb{R}.
$$
It is not difficult to see that $\phi$ is a flow of $X$.
If $\lambda=0$, then $\phi$ is the trivial flow.
If $\lambda\neq 0$, then
$Sing(\phi)=\{0,1\}$ and the remainder orbits go from $0$ to $1$ or viceversa depending on whether
$\lambda>0$ or $\lambda<0$. Then, $\phi$ is Morse-Smale and so it has the shadowing property but is not equicontinuous.
We shall prove that if $\lambda=1$, then $\phi$ satisfies the condition in Lemma \ref{bariloche}.

First note that
$dist(z,Sing(\phi))=\min\{z,1-z\}
$ for every $z\in X.$
It follows that $dist(z,Sing(\phi))=z$ (if $z\leq \frac{1}2$) or $1-z$ (otherwise).

Take $\epsilon>0$ and $0<\delta<\frac{1}2$ such that
$$
\ln\left(\frac{1+\delta}{1-\delta}\right)
\leq \epsilon.
$$
We will show that
$$
|x-y|\leq \delta\min\{x,1-x\} \quad \Rightarrow \quad
y\in \phi_{[-\epsilon,\epsilon]}(x).
$$
Notice that the left-hand side of the above implication is equivalent to
$$
x-\delta\min\{x,1-x\}\leq y\leq x+\delta\min\{x,1-x\}.
$$
Since $\delta<\frac{1}2$, one has $y\in X\setminus Sing(\phi)$.

Assume $x\leq \frac{1}2$.
It follows that
$$
(1-\delta)\leq \frac{y}x\leq (1+\delta)
\quad \mbox{ and } \quad
\frac{1-x}{1-x(1-\delta)}\leq \frac{1-x}{1-y}\leq \frac{1-x}{1-x(1+\delta)}.
$$
Since $0<x\leq \frac{1}2$, $0<\frac{x}{1-x}\leq1$ and then
$$
\frac{1-x}{1-x(1-\delta)}=\frac{(1-x)}{(1-x)+x\delta}=\frac{1}{1+\left(\frac{x}{1-x}\right)\delta}\geq \frac{1}{1+\delta}.
$$
Likewise,
$$
\frac{1-x}{1-x(1+\delta)}\leq \frac{1}{1-\delta}
$$
so
$$
\frac{1}{1+\delta}\leq \frac{1-x}{1-y}\leq \frac{1}{1-\delta}
$$
thus
$$
\ln\left(
\frac{1-\delta}{1+\delta}\right)\leq
\ln\left(\frac{y}x\cdot \frac{1-x}{1-y}\right)\leq
\ln\left(\frac{1+\delta}{1-\delta}\right).
$$
On the other hand, it follows from the definition of $\phi$ that
the equation $\phi_t(x)=y$
is solved by
$$
t=\ln\left(\frac{y}x\cdot \frac{1-x}{1-y}\right).
$$
Then, the choice of $\delta$ implies
$$
-\epsilon\leq t\leq \epsilon
$$ yielding
$y\in \phi_{[-\epsilon,\epsilon]}(x)$ for $x\leq \frac{1}2$.
Interchanging the roles of $x$ and $y$ above by $1-x$ and $1-y$ respectively we get $y\in \phi_{[-\epsilon,\epsilon]}(x)$ when $x>\frac{1}2$ too.
Therefore, $\phi$ satisfies the condition in Lemma \ref{bariloche},
and so, it is both singular-expansive and singular-equicontinuous.
This completes the proof.
\end{proof}

\begin{proof}[Proof of Theorem \ref{gasolina}]
First suppose that
$\Omega(\phi)\setminus Sing(\phi)$ is closed.
Then, $K=\Omega(\phi)\setminus Sing(\phi)$ is compact
contained in $X\setminus Sing(\phi)$ hence
$h(\phi,K)\leq h^*(\phi)$.
On the other hand, by well-known properties of the topological entropy (p. 403 in \cite{b1}),
$h(\phi)=h(\phi,X)=h(\phi,K\cup Sing(\phi))\leq \max\{h(\phi,K), h(\phi,Sing(\phi))\}=h(\phi,K)\leq h^*(\phi)$.
Since $h^*(\phi)\leq h(\phi)$, we are done.

Now suppose that $\phi$ is singular-equicontinuous.
Take $K\subset X\setminus Sing(\phi)$ compact and $\epsilon>0$.
For this $\epsilon$ let $\delta$ be given by the singular-equicontinuity of $\phi$.
Since $K\cap Sing(\phi)=\emptyset$,
$\delta dist(x, Sing(\phi))>0$ for every $x\in K$.
Then, since $K$ is compact,
there is $F\subset K$ finite such that
$$
K\subset \bigcup_{x\in F}B[x,\delta dist(x,Sing(\phi))].
$$
Now take $y\in K$ and $t>0$.
Then,
$$
d(x,y)\leq \delta dist(x,Sing(\phi))
$$
for some $x\in F$,
and so, by singular-equicontinuity,
$$d(\phi_l(x),\phi_l(y))\leq \epsilon\quad\quad\forall 0\leq l\leq t.
$$
It follows that $F$ is $(t,\epsilon)$-spanning for every $t>0$
thus
$r(t,\epsilon)\leq car(F)$ the cardinality of $F$ for every
$t>0$.
Therefore,
$$
\limsup_{t\to\infty}\frac{1}t\ln r(t,\epsilon)\leq \limsup_{t\to\infty} \frac{\ln car(F)}t=0, \quad\quad\forall \epsilon>0.
$$
Then, $h(\phi,K)=0$ for every compact subset $K\subset X\setminus Sing(\phi)$ so $h^*(\phi)=0$.
\end{proof}

\end{document}